\documentclass[12pt]{amsart}
\usepackage{amsmath,amssymb,amscd}
\usepackage[all]{xy}

\makeatletter
\@namedef{subjclassname@2010}{%
  \textup{2010} Mathematics Subject Classification}
\makeatother

\newtheorem{definition}{Definition}[section]
\newtheorem{theorem}[definition]{Theorem}
\newtheorem{proposition}[definition]{Proposition}
\newtheorem{lemma}[definition]{Lemma}
\newtheorem{corollary}[definition]{Corollary}
\newtheorem{remark}[definition]{Remark}
\newtheorem{example}[definition]{Example}
\numberwithin{equation}{section}

\def\RR{{\mathbb R}}
\def\NN{{\mathbb N}}
\def\ZZ{{\mathbb Z}}

\def\U{{\mathcal U}}
\def\V{{\mathcal V}}
\def\W{{\mathcal W}}

\def\B{{\mathrm B}}	
\def\CB{{\overline{\mathrm B}}}	
\def\mesh{\operatorname{mesh}}	
\def\L{\operatorname{Leb}}		
\def\mul{\operatorname{mult}}	
\def\hmul{\text{-}\operatorname{mult}}	
\def\ANdim{\dim_{\operatorname{AN}}}	
\def\asdim{\operatorname{asdim}}		
\def\ANasdim{\asdim_{\operatorname{AN}}}	
\def\diam{\operatorname{diam}}	
\def\Lip{\operatorname{Lip}}	
\def\id{\operatorname{id}}		

\begin{document}
\baselineskip=17pt

\title{Dimension-raising maps in a large scale}

\author[T. Miyata]{Takahisa Miyata}
\address{Department of Mathematics and Informatics,  Graduate School of Human Development and Environment,
Kobe University, Kobe, 657-8501 JAPAN}
\email[T. Miyata]{tmiyata@kobe-u.ac.jp}

\author[\v Z, Virk]{\v{Z}iga Virk}
\address{Faculty of Mathematics and Physics, University of Ljubljana, Jadranska 21, Ljubljana, 1000 SLOVENIA}
\email[\v{Z}. Virk]{virk@fmf.uni-lj.si}

\date{}

\begin{abstract}
Hurewicz's dimension-raising theorem of  states that
	for every $n$-to-$1$ map $f: X\rightarrow Y$, $\dim Y \leq \dim X + n$ holds.
In this paper we introduce a new notion of finite-to-one like map in a large scale setting.
Using this notion we formulate a dimension-raising type theorem
	for the asymptotic dimension and the asymptotic Assouad-Nagata dimension.
It is also well-known as Hurewicz's finite-to-one mapping theorem that
	$\dim X \leq n$ if and only if there exists an $(n+1)$-to-$1$ map from a $0$-dimensional space onto $X$.
We formulate a finite-to-one mapping type theorem for the asymptotic dimension and the asymptotic Assouad-Nagata dimension.
\end{abstract}



\maketitle
\section{Introduction}

Let us recall the classical Hurewicz dimension theorems for maps.

\begin{theorem}[Dimension-lowering theorem]\label{Hurewicz-lowering-1}
Let $f: X\rightarrow Y$ be a closed surjective map between metrizable spaces.
Then $\dim X \leq \dim Y + \dim f$,
	where $\dim f = \sup\{\dim f^{-1}(y): y\in Y\}$.
\end{theorem}

\begin{theorem}[Dimension-raising theorem]\label{Hurewicz-raising-2}
Let $f: X\rightarrow Y$ be a closed surjective map between metrizable spaces such that $|f^{-1}(y)| \leq n+1$ for each $y\in Y$.
Then $\dim Y \leq \dim X + n $.
\end{theorem}

\begin{theorem}[Finite-to-one mapping theorem]\label{Hurewicz-finite-to-one-3}
Let $X$ be a metrizable space.
Then $\dim X \leq n$ if and only if there exists a zero-dimensional metric space $Y$
	and a closed surjective map $f: Y \rightarrow X$ such that $|f^{-1}(x)| \leq n+1$ for each $x\in X$.
\end{theorem}

G. Bell and A. Dranishnikov \cite{Bell-Dranishnikov-1} proved the dimension-lowering theorem for the asymptotic dimension,
	and N. Brodskiy, J. Dydak, M. Levin, and A. Mitra \cite{Brodskiy-Dydak-Levin-Mitra-1} generalized it to
	the Assouad-Nagata dimension and the asymptotic Assouad-Nagata dimension.
However, there is no simple translation of  the dimension-raising theorem in a large scale setting
	since there are simple one-to-one dimension raising coarse maps.
	
In this paper  we introduce conditions, called (B)${}_n$ and (C)${}_n$, respectively,
	which correspond to the condition that a map is $n$-to-$1$.
Using those conditions we	formulate a dimension-raising type theorem and a finite-to-one mapping type theorem
		for the asymptotic dimension and the asymptotic Assouad-Nagata dimension.

Our main theorems for the asymptotic dimension $\asdim$ state:

\begin{theorem}\label{asymptotic-1}
Let $X$ and $Y$ be metric spaces, and let $f: X\rightarrow Y$ be a coarse, coarsely surjective map with the following property:
\begin{enumerate}
\item[(B)${}_n$]
For each $r < \infty$, there exists $d < \infty$ so that
	for each subset $B$ of $Y$ with $\diam(B) \leq r$, $f^{-1}(B) = \bigcup_{i=1}^n A_i$ for some subsets $A_i$ of $X$ with $\diam(A_i) \leq d$ for $i = 1, \ldots, n$.
\end{enumerate}
Then the following holds:
$$\asdim Y \leq (\asdim X + 1)\cdot n -1$$
\end{theorem}

\begin{theorem}\label{asymptotic-2}
Let $X$ be a metric space.
Then $\asdim X \leq n$ if and only if there exist a metric space $Y$ with $\asdim Y = 0$
	and a coarse map $f: Y \rightarrow X$ with condition (B)${}_{n+1}$.
\end{theorem}

Our main theorems for the asymptotic Assouad-Nagata dimension $\ANasdim$ state:

\begin{theorem}\label{asymptotic-AN-1}
Let $X$ and $Y$ be metric spaces, and let $f: X\rightarrow Y$ be an asymptotically Lipschitz, coarsely surjective map with the following property:
\begin{enumerate}
\item[(C)${}_n$]
There exist $c, d > 0$ so that
	for each $r < \infty$ and for each subset $B$ of $Y$ with $\diam(B) \leq r$,
	$f^{-1}(B) = \bigcup_{i=1}^n A_i$ for some subsets $A_i$ of $X$ with $\diam(A_i) \leq cr + d$ for $i = 1, \ldots, n$.
\end{enumerate}
Then the following holds:
$$\ANasdim Y \leq (\ANasdim X + 1)\cdot n -1$$
\end{theorem}

\begin{theorem}\label{asymptotic-AN-2}
Let $X$ be a metric space.
Then $\ANasdim X \leq n$ if and only if there exist a metric space $Y$ with $\ANasdim Y = 0$
	and an asymptotic Lipschitz map $f: Y \rightarrow X$ with condition {\rm (C)}${}_{n+1}$.
\end{theorem}

The ``if'' parts of Theorems \ref{asymptotic-2} and \ref{asymptotic-AN-2} immediately follow from Theorems  \ref{asymptotic-1} and \ref{asymptotic-AN-1},
	respectively.
For the ``only if '' parts of Theorems \ref{asymptotic-2} and \ref{asymptotic-AN-2},
	we introduce the notion of $n$-precode structure, which is a sequence of covers with some conditions
	determining a map with property {\rm (B)}${}_{n+1}$ or {\rm (C)}${}_{n+1}$ from an ultrametric space to the given space.

We give various examples of dimension-raising maps.
In particular, we present a simple example of $1$-precode structure for $(\ZZ, d_{\varepsilon})$ with the Euclidean metric $d_{\varepsilon}$.


A finite-to-one mapping theorem for the Asouad-Nagata dimension was obtained in  \cite{Miyata-Yoshimura-1},
	in which a condition called (B) (see Section 3) was introduced.
Using condition (B),
	we show a dimension-raising type theorem for the Assouad-Nagata dimension as well.

Throughout the paper, $\NN$, $\ZZ$, $\RR$, $\RR_{+}$ denote the set of nonnegative integers, the set of integers,
	the set of real numbers, and the set of positive real numbers, respectively.
For any set $X$, let $\id_X$ denote the identity map on $X$.

\section{Asymptotic dimension, Assouad-Nagata dimension, and asymptotic Assouad-Nagata dimension}

In this section, we recall the definitions and properties of the asymptotic dimension, the Assouad-Nagata dimension, and the asymptotic Assouad-Nagata dimension.
For more details, the reader is referred to
	\cite{Bell-Dranishnikov-1},  \cite{Bell-Dranishnikov-2}, \cite{Lang-Schlichenmaier-1}, and \cite{Dranishnikov-Smith-1}.

Let $(X, d)$ be a metric space.
For each $x\in X$ and $r>0$, let $\B(x, r) = \{y\in X: d(x, y) < r\}$, and $\CB(x, r) = \{y\in X: d(x, y) \leq r\}$.
For each subset $A$ of $X$, let $\diam(A)$ denote the diameter of $A$.

Let $\U$ be a cover of $X$.
The multiplicity of $\U$, in notation, $\mul(\U)$, is defined as the largest integer $n$
	so that no point of $X$ is contained in more than $n$ elements of $\U$,
	and the $r$-multiplicity of $\U$, in notation, $r\hmul(\U)$, is defined as the largest integer $n$
	so that no subset of diameter at most $r$ meets more than $n$ elements of $\U$.
The Lebesgue number of $\U$, in notation, $\L(\U)$, is defined as the supremum of positive numbers $r$
	so that for every subset $A$ with $\diam(A) \leq r$, there exists $U\in\U$ with $A\subset U$.
The mesh of $\U$, in notation, $\mesh(\U)$, is $\sup\{\diam(U): U\in \U\}$, and
	$\U$ is said to be uniformly bounded if $\mesh(\U) < \infty$.
A family $\U$ of subsets of $X$ is said to be $r$-disjoint
	if $d(x, x') > r$ for any $x$ and $x'$ that belong to different elements of $\U$.
	
A metric space $X$ is said to have the asymptotic dimension at most $n$, in notation, $\asdim X \leq n$,
	if there exists a function $D_X: \RR_{+} \rightarrow \RR_{+}$ (called an $n$-dimensional control function for $X$)
	so that for every $r<\infty$  there exist $r$-disjoint families $\U^0, \ldots, \U^n$ of subsets of $X$
	so that $\bigcup_{i=0}^n \U^i$ is a cover of $X$ and $\mesh(\U) \leq D_{X}(r)$.
	
A metric space $X$ is said to have the Assouad-Nagata dimension at most $n$, in notation, $\ANdim X \leq n$,
	if there exists an $n$-dimensional control function $D_{X}$ so that $D_{X} (r)  = c\cdot r$ for some $c \geq 0$.
	
A metric space $X$ is said to have the asymptotic Assouad-Nagata dimension at most $n$, in notation, $\ANasdim X \leq n$,
	if there exists an $n$-dimensional control function $D_{X}$ so that $D_{X}(r)  = c\cdot r + d$ for some $c, d \geq 0$.
		
We write $\asdim X = n$
	if $\asdim X \leq n$ and $\asdim X \not\leq n-1$,
	and write $\asdim X = \infty$ if $\asdim X \not\leq n$ for any nonnegative integer $n$.
Similarly, we define $\ANdim X = n$ and $\ANasdim X = n$.

The following characterizations of the asymptotic dimension, the Assouad-Nagata dimension, and the asymptotic Assouad-Nagata dimension are well-known
	(see \cite{Bell-Dranishnikov-2}, \cite{Lang-Schlichenmaier-1}, and \cite{Dranishnikov-Smith-1}).

\begin{proposition}\label{ASYM-0}
Let $X$ be a metric space.
Then the following conditions are equivalent:
\begin{enumerate}
\item
	$\asdim X \leq n$.
\item
	For every uniformly bounded cover $\V$ of $X$, there exists a uniformly bounded cover $\U$ of $X$
	so that $\mul(\U) \leq n+1$ and $\V < \U$.
\item
	For every $s < \infty$, there exists a uniformly bounded cover $\V$ of $X$ so that $s\hmul(\V) \leq n+1$.
\item
	For every $t < \infty$, there exists a uniformly bounded cover $\W$ of $X$ so that $\L(\W) \geq t$ and $\mul(\W) \leq n+1$.
\end{enumerate}
\end{proposition}

\begin{proposition}\label{ASYM-AN-0}\label{AN-0}
Let $X$ be a metric space.
Then the following conditions are equivalent:
\begin{enumerate}
\item
	$\ANasdim X \leq n$ (resp., $\ANdim X \leq n$).
\item
	There exists $c > 0$ (resp., exist $c, s_0 > 0$) so that
		for every $s <\infty$ (resp., $s\geq s_0$), there exists a cover $\V$ of $X$ with $\mesh(\V) \leq cs$ and $s\hmul(\V) \leq n+1$.
\item
	There exists $c > 0$ (resp., exist $c, t_0 > 0$) so that
		for every $t < \infty$ (resp., $t\geq t_0$), there exists a cover $\W$ of $X$ with $\mesh(\W) \leq ct$, $\L(\W) \geq t$, and $\mul(\W) \leq n+1$.
\end{enumerate}
\end{proposition}

The following characterization of the asymptotic dimension will be used in Section 6.

\begin{proposition}\label{ASYM-15}
Let $X$ be a metric space.
Then the following conditions are equivalent:
\begin{enumerate}
\item
	$\asdim X \leq n$.
\item
	For every $s < \infty$  and $t < \infty$,
		there exists a uniformly bounded cover $\U$ of $X$ so that $s\hmul(\U) \leq n+1$ and $\L(\U) \geq t$.
\end{enumerate}
\end{proposition}

\begin{proof}
The implication (2)$\Rightarrow$(1) is obvious by Proposition \ref{ASYM-0}.
To show the implication (1)$\Rightarrow$(2), suppose $\asdim X \leq n$.
Let $s < \infty$ and $t< \infty$.
Let $r \geq s + 4t$.
Then  by definition
	there exist uniformly bounded $r$-disjoint families $\U^0, \ldots, \U^n$ of subsets of $X$ so that
	$\U' = \bigcup_{i=0}^n \U^i$ is a cover of $X$.
Consider the cover $\U = \{ \B(U, 2t): U\in\U'\}$.
Then $s\hmul(\U) \leq n+1$.
Indeed, let $A$ be a subset of $X$ with $\diam (A) \leq s$ such that
	$A \cap \B(U, 2t) \neq \emptyset$ and $A \cap \B(U', 2t) \neq \emptyset$ for some $U, U' \in\U'$.
Then $d(U, U') \leq s + 4t \leq r$,
	which implies $U\in\U^{i}$ and $U'\in\U^{i'}$ for some $i, i'$ with $i\neq i'$.
Thus $A$ intersects with at most $n+1$ elements of $\U$,
	showing that $s\hmul(\U) \leq n+1$.
To show $\L(\U) \geq t$, let $A$ be a subset of $X$ so that $\diam(A) \leq t$.
Then $A\cap U \neq \emptyset$ for some $U\in\U'$,
	and hence $A \subset \B(U, 2t)$.
This shows (2).
\end{proof}

The following characterization of the asymptotic Assouad-Nagata dimension will be used in Section 7.

\begin{proposition}\label{ASYM-AN-16}
Let $X$ be a metric space.
Then the following conditions are equivalent:
\begin{enumerate}
\item
	$\ANasdim X \leq n$.
\item
	There exist $c, d > 0$ so that for every $s < \infty$ and $t < \infty$,
		there exists a cover $\U$ of $X$ so that
			$\U$ is $(c\cdot (s+4t) + d)$-bounded,
			$s\hmul(\U) \leq n+1$,
			and $\L(\U) \geq t$.
\end{enumerate}
\end{proposition}

\begin{proof}
(2)$\Rightarrow$(1) is obvious by Proposition \ref{ASYM-AN-0}.
(1)$\Rightarrow$(2) is proved by the same argument as in the proof of Proposition \ref{ASYM-15}.
Indeed, let $s < \infty$ and $t < \infty$.
Let $c, d > 0$ be constants as in the definition of $\ANasdim X \leq n$.
Without loss of generality, we can assume $c \geq 2$.
Put $r = s + 4t$, and
	let  $\U^0, \ldots, \U^n$ be $(cr /2 + d)$-bounded $r$-disjoint families of subsets of $X$ so that
	$\U' = \bigcup_{i=0}^n \U^i$ is a cover of $X$,
Then the cover $\U = \{ \B(U, 2t): U\in\U'\}$ satisfies the required conditions.
Note that $(c\cdot (s+4t) + d)$-boundedness of $\U$ follows from the following evaluations:
	$\mesh(\U) \leq \mesh(\U') + 4t \leq c \cdot (s + 4t) / 2+ d + 4t \leq c \cdot (s + 4t) + d$.
\end{proof}

\section{Dimension-raising maps: properties (B)${}_n$ and (C)${}_n$}

In this section, we prove dimension-raising type theorems
	for the Assouad-Nagata dimension, the asymptotic dimension, and the asymptotic Assouad-Nagata dimension.

A map $f: (X, d_X) \rightarrow (Y, d_Y)$ is said to be bornologous if there exists a function $\delta_f: \RR_{+} \rightarrow \RR_{+}$
	so that $d_Y (f(x), f(x')) \leq \delta_f (d_X(x, x'))$ for all $x, x'\in X$,
	and it is coarse if it is bornologous and proper. It is coarsely surjective if $f(X)$ is coarsely dense in $Y$, i.e., if there exists $R>0$ so that $d_Y(y,f(X)) \leq R$ for every $y\in Y$.
It is Lipschitz (resp., asymptotically Lipschitz)
	if there exists a function $\delta_f: \RR_{+} \rightarrow \RR_{+}$ so that $d_Y (f(x), f(x')) \leq \delta_f (d_X(x, x'))$ for all $x, x'\in X$
	and $\delta_f (t) = ct$ for some $c>0$ (resp., $\delta_f (t) = ct + b$ for some $b, c > 0$).
It is quasi-isometric
	if:
    \begin{enumerate}
      \item there exist functions $\delta_f, \gamma_f: \RR_{+} \rightarrow \RR_{+}$ so that
	$\gamma_f (d_X(x, x')) \leq d_Y (f(x), f(x')) \leq \delta_f (d_X(x, x'))$ for all $x, x'\in X$,
	$\delta_f (t) = ct + b$,  and $\gamma_f (t) = (1/c) \cdot t - b$ for some $b, c > 0$;
      \item $f$ is coarsely surjective.
    \end{enumerate}

Two maps $f, f^{\prime}: (X, d_X) \rightarrow (Y, d_Y)$ are said to be close if there exists $S > 0$ so that $d_Y(f(x), g(x)) \leq S$ for every $x\in X$.
A map $f: (X, d_X) \rightarrow (Y, d_Y)$ is called a coarse equivalence
	if there exists a coarse map $g: (Y, d_Y) \rightarrow (X, d_X)$ so that $f\circ g$ is close to $\id_Y$ and $g\circ f$ is close to $\id_X$.

\subsection{Dimension-raising type theorem for the Assouad-Nagata dimension}

For any map $f: X\rightarrow Y$ and for each $n\in\NN$, consider the following conditions \cite{Miyata-Yoshimura-1}:
\begin{enumerate}
\item[(B)]
	There exists $d > 0$ so that for each $r > 0$ and for each $B \subset Y$ with $\diam(B) \leq r$,
			there exists $A\subset X$ with $\diam(A) \leq dr$ and $f(A) = B$.
\end{enumerate}

\begin{lemma}\label{AN-1}
Let $f: X\rightarrow Y$ be a map, and let $\U$ be a cover of $X$.
If $|f^{-1}(y)| \leq n$ for each $y\in Y$, then
$$\mul(f(\U)) \leq \mul(\U)\cdot n.$$
\end{lemma}

\begin{proof}
Let $k = \mul(\U)$.
Suppose to the contrary that $\mul(f(\U)) > kn$.
Then there exist  $U_1, \ldots, U_{kn+1} \in \U$ so that there exists $y \in f(U_1) \cap \cdots \cap f(U_{kn+1})$.
So, there exit $x_i \in U_i$ for $i = 1, \ldots, kn+1$ so that  $y = f(x_1) =  \cdots = f(x_{kn+1})$.
Since $|f^{-1}(y)| \leq n$, then there exists at least $(k+1)$ indices $i_1, \ldots, i_{k+1} \in \{1, \ldots, kn+1\}$ so that
	$x_{i_1} = \cdots = x_{i_{k+1}}$,
	implying that $U_{i_1} \cap \cdots U_{i_{k+1}} \neq \emptyset$.
This contradicts to $\mul(\U) \leq k$.
\end{proof}

\begin{theorem}
Let $X$ and $Y$ be metric spaces, and let $f: X\rightarrow Y$ be a surjective Lipschitz map
	so that $|f^{-1}(y)| \leq n$ for each $y\in Y$, and it has property (B).
Then the following holds:
$$\ANdim Y \leq (\ANdim X + 1)\cdot n - 1.$$
\end{theorem}

\begin{proof}
Since the assertion is trivial if $\ANdim X = \infty$,
	assume $m = \ANdim X < \infty$.
Then there exists $c > 0$ so that for each $r > 0$ there exists a cover $\U_r$ of $X$ with $\mul(\U_r) \leq m + 1$,
	$\mesh(\U_r) \leq cr$, and $\L(\U_r) \geq r$.
Let $d > 0$ be as in (B), and for each $r > 0$, let $\V_r = f(\U_{dr})$.
Then Lemma \ref{AN-1} implies $\mul(\V_r) \leq \mul(\U_{dr})\cdot n$.
Since $f$ is Lipschitz, we have
	$\mesh(\V_r) \leq \Lip(f) \mesh(\U_{dr}) \leq \Lip(f)\cdot cdr$.
To show that $\L(\V_r) \geq r$,
	let $B$ be a subset of $Y$ so that $\diam (B) \leq r$.
Then (A) implies that there exists a subset $A$ of $X$ so that $\diam(A) \leq dr$ and $f(A) = B$.
Since $\L(\U_{dr}) \geq dr$, $A \subset U$ for some $U\in\U_{dr}$.
Hence $B = f(A) \subset f(U) \in \V_r$, showing that $\L(\V_r) \geq r$.
Thus we have shown that $\ANdim X \leq (m+1)\cdot n-1$.
\end{proof}

\subsection{Dimension-raising type theorem for the asymptotic dimension}
For any map $f: X\rightarrow Y$ and for each $n\in\NN$, consider the following conditions:

\begin{enumerate}
\item[(B)${}_n$]
For each $r < \infty$, there exists $d < \infty$ so that
	for each subset $B$ of $Y$ with $\diam(B) \leq r$, $f^{-1}(B) = \cup_{i=1}^n A_i$ for some subsets $A_i$ of $X$ with $\diam(A_i) \leq d$ for $i = 1, \ldots, n$.
\end{enumerate}

The following properties are useful in constructing maps with property (B)${}_n$
	in later sections.

\begin{proposition}\label{PropPropertyB-1}
Suppose $f\colon X \to Y$ is a coarse map satisfying property (B)${}_n$ and $g\colon Y \to Z$ is a coarse map satisfying property (B)${}_m$. Then $g f$ is a coarse map satisfying property (B)${}_{n \cdot m}$.
\end{proposition}

\begin{proof}
Let $A\subset Z$ be an $r$-bounded set. Then $g^{-1}(A)$ is a union of $d_g$-bounded sets $A_1,\ldots, A_m$. Similarly, for each $i$ the set $f^{-1}(A_i)$ union of $d_f$-bounded sets $A^i_1,\ldots, A^i_n$. Consequently, $(gf)^{-1}(A)$ is a union of $nm$-many $d_f$-bounded sets $\{A_i^j\}_{i=1,\ldots,m; j=1,\ldots,n}$.
\end{proof}

\begin{proposition}\label{PropPropertyB-2}
Suppose $f\colon X \to Y$ is a coarse map satisfying property (B)${}_n$ and $g\colon Z \to W$ is a coarse map satisfying property (B)${}_m$. Then $g \times f$ is a coarse map satisfying property (B)${}_{n \cdot m}$.
\end{proposition}

\begin{proof}
Suppose $p_Y$ and $p_W$ are projections of $Y\times W$ to $Y$ and $W$ respectively. Given an $r$-bounded set $A\subset Y \times W$, $p_Y(A)$ and $p_W(A)$ are $r$-bounded as well. Furthermore, since $f^{-1}(p_Y(A))$ is a union of $d_f$ bounded sets $A_1, \ldots, A_n$ and $g^{-1}(p_W(A))$ is a union of $d_g$ bounded sets $B_1, \ldots, B_m$ we get that $(f \times g)^{-1}(A)$ is a union of $m\cdot n$-many $(d_f+d_g)$-bounded sets $\{A_i \times B_j\}_{i=1,\ldots,n; j=1,\ldots,m}$.
\end{proof}

\begin{proposition}\label{PropPropertyB-3}
Suppose $X$ is a metric space of asymptotic dimension $0$ and $Y$ is any metric space. Then $\asdim (X \times Y)=\asdim Y$.
\end{proposition}

\begin{proof}
Let $n=\asdim Y$. Given $r<\infty$ there exist $d\in \RR$, $d$-bounded $r$-disjoint families $\mathcal{U}_0,\ldots, \mathcal{U}_n$ of subsets of $Y$ so that $\cup_{i=0}^n \mathcal{U}_i$ is a cover of $Y$ and a $d$-bounded $r$-disjoint $\mathcal{V}$ cover of $X$. Define $\mathcal{W}_i=\{V\times U\mid U\in \mathcal{U}_i, V\in \mathcal{V}\}$ for $i=0,\ldots,n$ and  note that $\mathcal{W}_0,\ldots, \mathcal{W}_n$ is a collection of $2d$-bounded $r$-disjoint families of subsets of $X\times Y$ so that $\bigcup_{i=0}^n \mathcal{W}_i$ is a cover of $X\times Y$. Hence $\asdim (X \times Y)\leq n=\asdim Y$. Since $X\times Y$ contains an isometric copy of $Y$ we also have $\asdim (X \times Y)\geq\asdim Y$.
\end{proof}

\begin{lemma}\label{ASYM-1}
Let $f: X\rightarrow Y$ be a map, and let $\U$ be a cover of $X$.
Suppose that $f$ satisfies condition {\rm (B)}${}_n$.
Let $r < \infty$, and let $d < \infty$ be as in {\rm (B)}${}_n$.
Then
	$$r\hmul (f(\U)) \leq d\hmul(\U)\cdot n.$$
\end{lemma}

\begin{proof}
Let $m = d\hmul(\U)$.
Suppose to the contrary that $r\hmul (f(\U)) > mn$.
Then there exists a subset $B$ of $Y$ with $\diam(B) \leq r$
	so that $B\cap f(U_i) \neq \emptyset$ for some $U_1, \ldots, U_{mn+1} \in\U$.
Then (B)${}_n$ implies that $f^{-1}(B) = \cup_{j=1}^n A_j$ for some subsets $A_j$ of $X$ with $\diam(A_j) \leq d$ for $i = 1, \ldots, n$.
So, $\emptyset \neq f^{-1}(B) \cap U_i = (\cup_{j=1}^n A_j) \cap U_i$ for $i = 1, \ldots, mn+1$.
This implies that there exists $j_0$ so that $A_{j_0} \cap U_i \neq \emptyset$ for some $i \in \{ i_1, \ldots, i_{m+1} \} \subset \{1, \ldots, mn + 1\}$.
This contradicts to the condition that $d\hmul(\U) = m$.
\end{proof}

\begin{theorem}
Let $X$ and $Y$ be metric spaces, and let $f: X\rightarrow Y$ be a coarse, coarsely surjective map with property {\rm (B)}${}_n$.
Then the following holds:
$$\asdim Y \leq (\asdim X + 1)\cdot n -1$$
\end{theorem}

\begin{proof}
Since the assertion is trivial if $\asdim X = \infty$, assume
	$m = \asdim X < \infty$.
Let $r > 0$, and let $d > 0$ be as in (B)${}_n$.
Then, by Proposition \ref{ASYM-0} (3), there exists a uniformly bounded cover $\U_d$ of $X$ so that $d\hmul(\U_d) \leq m+1$.
Consider $\V = f(\U_d)$.
By Lemma \ref{ASYM-1}, $r\hmul(\V) \leq d\hmul(\U_d)\cdot n \leq (m+1)n$.
Since $f$ is bornologous, $\V$ is uniformly bounded.
Hence $\asdim Y \leq (m+1)n -1 = (\asdim X + 1)\cdot n -1$, as required.
\end{proof}

\subsection{Dimension-raising type theorem for the asymptotic Assouad-Nagata dimension}
We can modify the argument for the asymptotic dimension to obtain the dimension-raising theorem for the asymptotic Assouad-Nagata dimension.

For any map $f: X\rightarrow Y$ and for each $n\in\NN$, consider the following conditions:

\begin{enumerate}
\item[(C)${}_n$]
There exist $c, r_0 > 0$ so that
	for each $r \geq r_0$ and for each subset $B$ of $Y$ with $\diam(B) \leq r$,
	$f^{-1}(B) = \cup_{i=1}^n A_i$ for some subsets $A_i$ of $X$ with $\diam(A_i) \leq cr$ for $i = 1, \ldots, n$.
\end{enumerate}

\begin{remark}\label{RemPropertyC}\rm
It can be verified that Propositions \ref{PropPropertyB-1}, \ref{PropPropertyB-2}, \ref{PropPropertyB-3} hold for the asymptotic Assouad-Nagata dimension
	if coarse map is replaced by asymptotic Lipschitz map.
\end{remark}

\begin{lemma}\label{ASYM-AN-1}
Let $f: X\rightarrow Y$ be a map, and let $\U$ be a cover of $X$.
Suppose that $f$ satisfies condition {\rm (C)}${}_n$.
Let $c > 0$ and $r_0 > 0$ be as in {\rm (C)}${}_n$.
Then for each $r \geq r_0$,
	$$r\hmul (f(\U)) \leq cr\hmul(\U)\cdot n.$$
\end{lemma}

\begin{proof}
We can use the same technique as in the proof of Lemma \ref{ASYM-1} to prove the assertion.
\end{proof}

\begin{theorem}
Let $X$ and $Y$ be metric spaces, and let $f: X\rightarrow Y$ be an asymptotically Lipschitz coarsely surjective map with property {\rm (C)}${}_n$.
Then the following holds:
$$\ANasdim Y \leq (\ANasdim X + 1)\cdot n -1$$
\end{theorem}

\begin{proof}
We can assume $m = \ANasdim X < \infty$.
Let $c, r_0 > 0$ be as in (B)${}_n$, and let $r \geq r_0$.
Then, by Proposition \ref{ASYM-AN-0} (2), there exists a cover $\U_r$ of $X$ so that $\mesh(\U_r) \leq cr$ and $d\hmul(\U_r) \leq m+1$.
Consider $\V = f(\U_r)$.
By Lemma \ref{ASYM-AN-1}, $r\hmul(\V) \leq cr\hmul(\U_r)\cdot n \leq (m+1)n$.
Since $f$ is asymptotically Lipschitz, $\mesh(\V) \leq c' cr + b$ for some $b, c' > 0$.
If $r\geq \max\{r_0, b/c\}$, then $\mesh(\V) \leq c'' cr$, where $c'' = c' + 1$.
Hence $\ANasdim Y \leq (m+1)n -1 = (\ANasdim X + 1)\cdot n -1$, as required.
\end{proof}


\section{$n$-precode structure for the asymptotic dimension}

A metric space $(X, d)$ is said to be ultrametric if $d(x, z) \leq \max\{d(x,y), d(y, z)\}$ for all $x, y, z\in X$.
Every ultrametric space has asymptotic dimension $0$.
Indeed, for each $r < \infty$, there exists an $r$-disjoint cover  $\U$ which consists of $r$-components.
Since each $r$-component of a ultrametric space is an $r$-ball, $\U$ is uniformly bounded.

In this section, we present a procedure to construct coarse maps from ultrametric spaces with property (B)${}_n$.

\begin{theorem}\label{ThmUltrametricCover}
Suppose $\mathcal{U}_0,\mathcal{U}_1, \ldots$ is a sequence of uniformly bounded covers of a metric space $X$ and fix $n\in \NN$.
    \begin{enumerate}
      \item If for every $i$ and every $U\in \mathcal{U}_i$ there exists exactly one $V\in \mathcal{U}_{i+1}$ satisfying $U\subset V$ then every $W^0\in\mathcal{U}_0$ defines a unique sequence $(W^0,W^1,\ldots)$ with $W^i\in \mathcal{U}_i$ and $W^i\subset W^{i+1}.$
      \item Assume the conditions of the previous case along with the following additional condition: for every bounded subset $D\subset X$ there exist $i$ and $U\in \mathcal{U}_i$ so that $D\subset U$. Then the following rule defines an ultrametric on $\mathcal{U}_0$: $d_B(V,V)=0$ and for $V\neq W$
            $$
            d_B(V,W)= 3^{p(V,W)}\quad \emph{ where }p(V,W)={\min \{k\in \ZZ\mid \exists \widetilde U\in \mathcal{U}_k : V\cup W \subset \widetilde U\}}
            $$
            Furthermore, $\asdim(\mathcal{U}_0,d_B)=0$ and a map $q\colon \mathcal{U}_0\to X$ sending $U\in \mathcal{U}_0$ to any chosen point  $x\in U$ is coarse.
      \item Assume the conditions of the previous case along with the following additional condition: for every $r<\infty$ there exists $i$ so that $r\hmul(\U_{i}) \leq n$. Then $q$ satisfies condition {\rm (B)}$_n$.
    \end{enumerate}
\end{theorem}

\begin{proof}
(1) is obvious.

(2): The distance $d_B$ is finite (as the union of every pair of elements of $\mathcal{U}_0$ is contained in some $U\in \mathcal{U}_i$), symmetric and equals $0$ exactly for the case of identical elements of $\mathcal{U}_0$. It is easy to see that the uniqueness of sequences of (1) implies that $d_B$ is an ultrametric.

It has been remarked that the asymptotic dimension of an ultrametric space is $0$. To see that $q$ is coarse observe that if $d_B(U,V)\leq 3^n$ then $d(q(U),q(V))\leq \mesh(\U_n)$.

(3): Fix $r<\infty$ and choose $i$ so that $r$-multiplicity of $\mathcal{U}_{i}$ is at most $n$. Suppose $B\subset Y$ is of diameter at most $r$ and let $U_1, \ldots, U_n$ denote the collection of all elements of $\mathcal{U}_{i}$ that have a nonempty intersection with $B$ (some elements may be identical since $B$ might intersect less than $n$-many elements of $\mathcal{U}_{i}$). Then $q^{-1}(B)$ is a union of sets $A_j=\{U\in \mathcal{U}_{0}\mid U\subset U_j\}$, which are of diameter at most $3^i$.

As an important technical detail we mention the following: if $U^0\in \mathcal{U}_{0}$ has a nonempty intersection with $B$ then (using the convention of (1)) $U^i$ contains $U^0$ hence is listed as the sets $U_j$ for some $j$. In particular, $U^0\in A_j$.
\end{proof}

\begin{definition}\rm
Any sequence of uniformly bounded covers satisfying (1)--(3) of Theorem \ref{ThmUltrametricCover} is called the $n$-precode structure for asymptotic dimension.
\end{definition}

\begin{corollary}
If a metric space $X$ admits an $n$-precode structure for asymptotic dimension then there exists an ultrametric space $Z$ and a coarse map $f\colon Z\to X$ with property {\rm (B)}${}_n$.
\end{corollary}

\begin{corollary}\label{Cor-1-precode}
If a metric space $X$ admits a $1$-precode structure for asymptotic dimension then there exists an ultrametric space $Z$ and a coarse equivalence $f\colon Z\to X$.
\end{corollary}

\begin{proof}
Suppose $X$ admits a $1$-precode structure $\U_0, \U_1, \ldots$.
Let $Z = \U_0$ and let $f: Z \rightarrow X$ be the coarse map with property (B)${}_1$ defined as in (2) of Theorem \ref{ThmUltrametricCover}.
To verify that $f$ is a coarse equivalence, we define a map $g: X\rightarrow Z$ by $g(x) = U_x$ for each $x\in X$
	where $U_x$ is an element of $ \U_0$ with $x\in U$.

To show that $g$ is bornologous,
	let $R < \infty$, and let $d(x, y) < R$.
Take $k\in\ZZ$ so that $R \leq 3^k$.
Then $g(x) = U_x$ and $g(y) =U_y$,
	where $U_x$ and $U_y$ are elements of $\U_0$ with $x \in U_x$ and $y\in U_y$, respectively.
Condition (3) of Theorem \ref{ThmUltrametricCover} implies that
	$3^k\hmul(\U_i) \leq 1$ for some $i\in\NN$.
Condition (1) of Theorem \ref{ThmUltrametricCover} implies that
	there exist unique elements $U_x^{\prime}$ and $U_y^{\prime}$ of $\U_i$ so that $U_x \subset U_x^{\prime}$ and $U_y \subset U_y^{\prime}$.
Since $d(x, y) < 3^k$ and $3^k\hmul(\U_i) \leq 1$, $U_x^{\prime} = U_y^{\prime}$.
This means that $d_B (U_x, U_y) \leq 3^i$.

To verify that $g$ is proper, let $R< \infty$.
Suppose $A$ is a subset of $Z$ so that $\diam(A) \leq R$, and take $k\in\ZZ$ so that $R \leq 3^k$.
Let $x, y\in g^{-1}(A)$.
Then $g(x) = U_x$ and $g(y) =U_y$,
	where $U_x$ and $U_y$ are elements of $\U_0$ with $x \in U_x$ and $y\in U_y$, respectively.
Since $d_B(U_x, U_y) \leq 3^k$, $d(x, y) \leq \mesh(\U_{k})$,
	showing that $\diam g^{-1}(A) \leq \mesh(\U_{k})$.

To show that $f\circ g$ are close to $\id_X$,
	let $x\in X$.
Then $g(x) = U_x$, where $U_x$ is an element of $\U_0$ so that $x \in U_x$,
	and so $f(g(x)) \in U_x$.
This means that $d(f(g(x)), x) \leq \mesh(\U_0)$.
Also $g\circ f = \id_Z$.
This shows that $Z$ and $X$ are coarse equivalence.	
\end{proof}

\begin{example}\label{ASYM-EXAMPLE-2}
\rm
The metric space $(\NN, d_{\varepsilon})$ where $d_{\varepsilon}$ is the Euclidean metric admits a $2$-precode structure for asymptotic dimension.
Indeed, we define $\U_0 = \{U_n^0: n \in\NN\}$, where $U_n^0 = \{n\}$ for each $n\in \ZZ$.
Assuming that $\U_i = \{U_n^i~ |~ n \in\NN\}$ has been defined,
	we define $\U_{i+1} = \{U_n^{i+1} ~|~ n \in\NN\}$, where
		$U_n^{i+1} = U_{2n}^i \cup U_{2n+1}^i$ for each $n\in \NN$.
Thus defined sequence of covers $\U_i$ satisfies conditions (1) -- (3) of Theorem \ref{ThmUltrametricCover}.

Hence there exist an ultrametric space $(X, d)$ and a coarse map $f: (X, d) \rightarrow (\NN, d_{\varepsilon})$ with property (B)${}_2$.
Note $\asdim X = 0$ and $\asdim  (\NN, d_{\varepsilon}) = 1$.

Proposition \ref{PropPropertyB-2} implies that
	$f\times \id_{\NN^n}: (X, d)\times (\NN^n, d_{\varepsilon}) \rightarrow (\NN^{n+1},  d_{\varepsilon})$ is a coarse map with property (B)${}_2$.
Note that $\asdim X\times \NN^n = \asdim \NN^n = n$ (Proposition \ref{PropPropertyB-3}) and $\asdim  (\NN^{n+1},  d_{\varepsilon}) = n+1$.
\end{example}

\begin{example}
\rm

In this example we present a $2$-precode structure for asymptotic dimension on the metric space $(\NN, d_{\varepsilon})$ where $d_{\varepsilon}$ is the Euclidean metric. The example is closely related to  Example \ref{ASYM-EXAMPLE-2} (and analogous conclusions can be easily drawn) although the formal description is somewhat different. Define $a^k(n)=\{n,n+1,\ldots,n+3^k-1\}\subset \ZZ$. The $2$-precode structure for asymptotic dimension is given by covers $\U_k=\{a^k(n)\mid \exists j\in \ZZ : n=(3^{k+1}-1)\cdot \frac{1}{2}+j \cdot 3^k\}$.

Note that $\U_k$ is a cover of $\ZZ$ by disjoint intervals of length $3^k$, the element $0$ being approximately in the middle of one such interval. Cover $\U_{k+1}$ is obtained by taking unions of three consecutive intervals so that the obtained cover is disjoint and so that element $0$ is approximately an the middle of one such union (i.e., three times larger interval).
\end{example}

\section{Finite-to-one mapping theorem for the asymptotic dimension}

In this section, using the $n$-precode structure, we prove a finite-to-one mapping type theorem for the asymptotic dimension.

\begin{theorem}\label{ThmAsdimPrecode}
Let $X$ be a metric space.
If $\asdim X \leq n$ then $X$ admits an  $(n+1)$-precode structure for asymptotic dimension.
\end{theorem}

\begin{proof}
We provide an inductive construction of covers $\mathcal{U}_i$. Fix $x_0\in X$ and let $\mathcal{U}_0=\{\{x\}\}_{x\in X}$ be a cover by singletones.

Let $k\in \NN$ and suppose we have constructed covers $\mathcal{U}_0, \ldots, \mathcal{U}_{k}$ with the following properties:
\begin{enumerate}
  \item $\mathcal{U}_i$ is an $M_i$-bounded cover, $\forall i=0, \ldots, k$;
  \item $i$-multiplicity of $\mathcal{U}_i$ is at most $n+1$, $\forall i=0, \ldots, k$;
  \item elements of $\mathcal{U}_i$ are disjoint, $\forall i=0, \ldots, k$;
  \item given $i<k$ and $U\in \mathcal{U}_i$ there exists $V\in \mathcal{U}_{i+1}$ containing $U$ (such element is unique by the previous property);
  \item given $i<k$ there exists $U_{\alpha_i}\in \mathcal{U}_i$ containing the closed ball $B(x_0,i)$ (again, such element is unique by (3)).
\end{enumerate}

Cover $\mathcal{U}_{k+1}$ is constructed as follows. By Proposition \ref{ASYM-15} there exists an $N_{k+1}$-bounded cover $\mathcal{V}_{k+1}=\{V_\beta\}_{\beta\in \Sigma}$ of $(k+1+2 M_k)$-multiplicity at most $n+1$ and of Lebesgue number at least $2(k+1)$. Let $V_{\alpha_{k+1}}\in \mathcal{V}_{k+1}$ be a set containing the closed ball $B(x_0,k+1)$. For every $U\in \mathcal{U}_k$ define index $\tau(U)\in \Sigma$ in the following way:
\begin{itemize}
  \item if $U\cap V_{\alpha_{k+1}}\neq \emptyset$ then $\tau(U)=\alpha_{k+1}$;
  \item else choose $\tau(U)$ to be any index in $\Sigma$ so that $U\cap V_{\tau(U)}\neq \emptyset$.
\end{itemize}
Define $\mathcal{U}_{k+1}=\{U_\beta\}_{\beta\in \Sigma}$ where
$$
U_\beta=\bigcup_{W\in \mathcal{U}_k, \tau(W)=\beta} W.
$$
The following is the verification that cover $\mathcal{U}_{k+1}$ satisfies the required conditions:
\begin{enumerate}
  \item cover $\mathcal{U}_{k+1}$ is $(2 M_k + N_{k+1})$-bounded by construction;
  \item $(k+1)$-multiplicity of $\mathcal{U}_{k+1}$ is at most $n+1$ (this is the consequence of two facts: for every $\beta\in \Sigma$ the $M_k$-neighborhood of $V_\beta$ contains $U_\beta$; and $(k+1+2 M_k)$-multiplicity of $\mathcal{V}_{k+1}$ is at most $n+1$.);
  \item elements of $\mathcal{U}_{k+1}$ are disjoint by construction as the elements of $\mathcal{U}_k$ are disjoint and each $U\in \mathcal{U}_{k}$ is assigned exactly one $\tau(U)$;
  \item obviously $U\subset U_{\tau(U)}$ for every $U\in \mathcal{U}_{k}$;
  \item $U_{\alpha_{k+1}}\in \mathcal{U}_{k+1}$ contains the closed ball $B(x_0,k+1)$ by construction.
\end{enumerate}

It is apparent from the properties listed above that the covers $\mathcal{U}_{i}$ form an  $(n+1)$-precode structure for asymptotic dimension  structure on $X$.
\end{proof}

Corollary \ref{CorFinite-to-oneMapping} is a large scale version of finite-to-one mapping theorem.

\begin{corollary}\label{CorFinite-to-oneMapping}
For every metric space $X$,
$\asdim X \leq n$ if and only if  there exist a metric space $Y$ of $\asdim Y = 0$ and a coarse map $q\colon Y\to X$ with {\rm (B)}${_{n+1}}$.
\end{corollary}

\begin{corollary}
For every $n\in \NN$ and $m\geq n$ there exist metric spaces $X$ and $Y$ with $\asdim Y=m$ and $\asdim X =n+m$, respectively, and
	a coarse map $q\colon Y\to X$ with property  {\rm (B)}${_{n+1}}$.
\end{corollary}

\begin{corollary}\label{CorAsdimUltrametric}
For every metric space $(X, d)$,
	$\asdim (X, d) = 0$ if and only if there exists an ultrametric $\rho$ on $X$ so that $\id: (X, d) \rightarrow (X, \rho)$ is a coarse equivalence.
\end{corollary}

\begin{proof}
The corollary easily follows from
	Theorems \ref{ThmAsdimPrecode} and \ref{Cor-1-precode}.
\end{proof}

Corollary \ref{CorAsdimUltrametric} generalizes the result
	by Brodskiy, Dydak, Levin, and Mitra \cite{Brodskiy-Dydak-Levin-Mitra-1}, which states that
		$\ANdim (X, d) = 0$
		if and only if  there is a ultrametric $\rho$ so that
		the identity map $\id: (X, d) \rightarrow (X, \rho)$ is bi-Lipschitz.

\section{Finite-to-one mapping theorem for the asymptotic Assouad-Nagata dimension}

In this section, we generalize the results in Sections 5 and 6 to the asymptotic Assouad-Nagata dimension.
The following is an analogue of Theorem \ref{ThmUltrametricCover}
	 which provides a general way to construct asymptotically Lipschitz maps from ultrametric spaces with property (C)${}_n$.

\begin{theorem}\label{ThmUltrametricCoverAsdimAN}
Suppose $\mathcal{U}_0,\mathcal{U}_1, \ldots$ is a sequence of uniformly bounded covers of a metric space $X$
	which satisfies conditions (1) and (2) in Theorem \ref{ThmUltrametricCover},
	 and fix $n\in \NN$.

    \begin{enumerate}
      \item
      	Assume the following condition:
		there exist $a > 1$ and $i_0\in\NN$ so that $\mesh(\U_i) \leq a^i$ for $i\geq i_0$.
	Then there exists an ultrametric $d_C$ on $\U_0$ so that a map $q: \U_0 \rightarrow X$ sending $U\in\U_0$ to any chosen point $x\in U$
		is asymptotically Lipschitz.
      \item Assume the condition of the previous case along with the following additional condition:
		there exist $c, r_0 > 0$ so that for every $r \geq r_0$ there exists $i\in\NN$
		so that $a^i \leq cr$ and $r\hmul(\U_i) \leq n$.
	Then $q$ satisfies condition {\rm (C)}$_n$.
    \end{enumerate}
\end{theorem}

\begin{proof}
(1): Let $d_C$ be the ultrametric $d_B$ obtained in Theorem \ref{ThmUltrametricCover} (2) with the base number $3$ being replaced by $a$, i.e., $d_C(V,V)=0$ and for $V\neq W$
          $$
            d_C(V,W)= a^{p(V,W)}\quad \emph{ where }p(V,W)={\min \{k\in \ZZ\mid \exists \widetilde U\in \mathcal{U}_k : V\cup W \subset \widetilde U\}}.
            $$

To see that $q$ is asymptotically Lipschitz,
	observe if $d_C(U,V) = a^n$ then $d(q(U),q(V))\leq \mesh(\U_n) \leq d_C(U,V)  + a^{i_0}$.

(2): Let $c, r_0 > 0$ be as in the hypothesis.
Fix $r \geq r_0$, and choose $i$ so that $a^i \leq cr$ and $r\hmul(\U_i) \leq n$.
Suppose $B\subset Y$ is of diameter at most $r$,
	and let $U_1, \ldots, U_n$ denote the collection of all elements of $\mathcal{U}_{i}$ that have a nonempty intersection with $B$.
Then $q^{-1}(B)$ is a union of sets $A_j=\{U\in \mathcal{U}_{0}\mid U\subset U_j\}$, which have $\diam(A_j) \leq a^i \leq cr$.
\end{proof}

\begin{definition}\rm
Any sequence of uniformly bounded covers satisfying (1)-(2) of Theorem \ref{ThmUltrametricCoverAsdimAN}
	is called the $n$-precode structure for asymptotic Assouad-Nagata dimension.
\end{definition}

The following is an analogue of Theorem \ref{ThmAsdimPrecode} for the asymptotic Assouad-Nagata dimension.

\begin{theorem}\label{ThmAsdimANPrecode}
Let $X$ be a metric space.
If $\ANasdim X \leq n$ then $X$ admits an  $(n+1)$-precode structure for asymptotic Assouad-Nagata dimension.
\end{theorem}

\begin{proof}
We inductively construct covers $\mathcal{U}_i$ which satisfy all the required conditions in Theorem \ref{ThmUltrametricCoverAsdimAN}.
Their constructions follow the steps used for Theorem \ref{ThmAsdimPrecode}.

Fix $x_0\in X$ and let $\mathcal{U}_0=\{\{x\}\}_{x\in X}$ be a cover by singletones.

Proposition \ref{ASYM-AN-16} implies that there exist $c, d > 0$ so that for each $s < \infty$ and $t < \infty$
	there exists a cover $\U_{s, t}$ of $X$ with $\mesh(\U_{s, t}) \leq c\cdot (s+ 4t) + d$, $s\hmul(\U_{s, t}) \leq n+1$, and $\L(\U_{s, t}) \geq t$.
Without loss of generality, we can assume $c \geq d \geq  2$.

Let $k\in \NN$ and suppose we have constructed covers $\mathcal{U}_0, \ldots, \mathcal{U}_{k}$ with the following properties:
\begin{enumerate}
  \item
  	$\mesh(\U_i) \leq (14c)^i$, $\forall i=0, \ldots, k$;
  \item
  	$((3^{i} - 1)/3)\hmul(\U_i) \leq n+1$, $\forall i=0, \ldots, k$;
  \item
  	elements of $\U_i$ are disjoint, $\forall i=0, \ldots, k$;
  \item
  	given $i<k$ and $U\in \U_i$ there exists a unique $V\in \U_{i+1}$ containing $U$;
  \item
  	given $i<k$ there exists a unique $U_{\alpha_i}\in \U_i$ containing $B(x_0, (3^{i} -1)/3)$.
\end{enumerate}

To define cover $\U_{k+1}$,
	let $\V_{k+1} = \{V_\beta\}_{\beta\in \Sigma}$ be the cover $\U_{s, t}$, where $s = 3^k + 2\cdot (14c)^k$ and $t = 2\cdot 3^{k}$.
Then $\V_{k}$ satisfies the following conditions:

	\begin{eqnarray}
 	&&	\mesh(\V_{k+1}) \leq c\cdot (3^{k+2} + 2\cdot (14c)^k) + d, 	\label{e7.1} \\
	&&	(3^{k} + 2\cdot (14c)^k)\hmul(\V_{k+1}) \leq n+1, 	 \label{e7.2} \\
	&&	\L(\V_{k+1}) \geq 2\cdot 3^{k}. 			\label{e7.3}
	\end{eqnarray}

Note that (\ref{e7.1}) holds since $s + 4t = 3^{k+2} + 2\cdot (14c)^k$.

Let $V_{\alpha_{k+1}}\in \mathcal{V}_{k+1}$ be a set containing $B(x_0, (3^{k+1} - 1)/3)$.
For every $U\in \mathcal{U}_k$ define index $\tau(U)\in \Sigma$ in the following way:
\begin{itemize}
\item
  	if $U\cap V_{\alpha_{k+1}}\neq \emptyset$ then $\tau(U)=\alpha_{k+1}$;
\item
	else choose $\tau(U)$ to be any index in $\Sigma$ so that $U\cap V_{\tau(U)}\neq \emptyset$.
\end{itemize}
Define $\U_{k+1}=\{U_\beta\}_{\beta\in \Sigma}$ where
$$
U_\beta=\bigcup_{W\in \U_k, \tau(W)=\beta} W.
$$
We claim that $\U_{k+1}$ satisfies the following conditions:
\begin{enumerate}
\item
	$\mesh(\U_{k+1}) \leq (14c)^{k+1}$;
\item
	$((3^{k+1} - 1)/3)\hmul(\U_{k+1}) \leq n+1$;
\item
	elements of $\U_{k+1}$ are disjoint;
\item
	$U\subset U_{\tau(U)}$ for every $U\in \mathcal{U}_{k}$;
\item
	$U_{\alpha_{k+1}}\in \mathcal{U}_{k+1}$ contains $B(x_0, (3^{k+1} - 1)/3)$.
\end{enumerate}
To see (1), observe
$$
	\begin{array}{ll}
	\mesh(\U_{k+1})	&
		\leq 	2\mesh(\U_{k}) + \mesh(\V_{k+1})		\\	
		&
		\leq 	2\cdot (14c)^k + c\cdot (3^{k+2} + 2\cdot (14c)^k) + d		\\
		&
		=	(2\cdot (14c)^k + 3^{k+2}\cdot c + d) + 2\cdot 14^k\cdot c^{k+1}	\\
		&
		\leq	14^k \cdot (2c^k + 3^2\cdot c + c) + 2\cdot 14^k \cdot c^{k+1}	\\
		&	
		\leq	14^k \cdot 12c^{k+1} + 2\cdot 14^k \cdot c^{k+1}			
		= 	(14c)^{k+1}.
	\end{array}
$$
Condition (2) follows from (\ref{e7.2}) and $(3^{k+1} - 1)/3 < 3^k$.
All the other conditions follow from constructions.
\end{proof}

\begin{corollary}\label{CorAsdimAN-1}
For every metric space $X$,
	$\ANasdim X \leq n$ if and only if there exist a metric space $Y$ of $\ANasdim Y = 0$ and an asymptotically Lipschitz map $q\colon Y\to X$
	with property {\rm (C)}${_{n+1}}$.
\end{corollary}

\begin{corollary}
For every $n\in \NN$ and $m\geq n$ there exist metric spaces $X$ and $Y$ with $\ANasdim Y=m$ and $\ANasdim X =n+m$, respectively, and
	an asymptotically Lipschitz map $q\colon Y\to X$ with property  {\rm (C)}${_{n+1}}$.
\end{corollary}

\begin{corollary}\label{CorAsdimAN-2}
If a metric space $X$ admits a $1$-precode structure for asymptotic Assouad-Nagata dimension
	then there exists an ultrametric space $Z$ and a quasi-isometric map $f\colon Z\to X$.
\end{corollary}

\begin{proof}
Let $\mathcal{U}_0,\mathcal{U}_1, \ldots$ be a $1$-precode structure, and
	let  $f: Z \rightarrow X$ be the asymptotic Lipschitz map defined as in Theorem \ref{ThmUltrametricCoverAsdimAN}.
It suffices to show that $f$ is a quasi-isometry.
Let $U, V \in \U_0, U\neq V$.
Let $n\in\NN$ be such that $a^{n-1} \leq d(f(U), f(V)) \leq a^n$.
Let $c, r_0 > 0$ be as in condition (2) of Theorem \ref{ThmUltrametricCoverAsdimAN}.
Then there exists  $i\in\NN$ so that $a^i \leq c\cdot (a^n + r_0)$ and $a^n \hmul(\U_i) \leq 1$.
Let $U^{\prime}$ and $V^{\prime}$ be the unique elements of $\U_i$ so that $U \subset U^{\prime}$ and $V \subset V^{\prime}$, respectively.
Then $U = U^{\prime}$ and $V = V^{\prime}$.
This implies that $d_C(U, V) \leq a^i \leq c \cdot a^n + c\cdot r_0 \leq (c\cdot a)\cdot d(f(U), f(V)) + c\cdot r_0$.
This shows that $f$ is quasi-isometric since the image of $f$ is apparently coarsely dense.
\end{proof}

\begin{corollary}\label{CorAsdimAN-3}
For every metric space $(X, d)$,
	$\ANasdim (X, d) = 0$ if and only if there exists an ultrametric $\rho$ on $X$ so that $\id: (X, d) \rightarrow (X, \rho)$ is a quasi-isometric map.
\end{corollary}

\begin{proof}
The corollary easily follows from
	Theorem \ref{ThmAsdimANPrecode} and Corollary \ref{CorAsdimAN-2}.
\end{proof}

\end{document}